\documentclass[11pt,letterpaper]{amsart}

\usepackage[letterpaper,top=3cm, bottom=3cm, left=3.5cm, right=3.5cm]{geometry}

\usepackage{amssymb}
\usepackage{enumerate}

\newtheorem{theorem}{Theorem}[section]
\newtheorem{proposition}[theorem]{Proposition}

\theoremstyle{definition}
\newtheorem{definition}[theorem]{Definition}
\newtheorem{question}[theorem]{Question}
\newtheorem{remark}[theorem]{Remark}
\newtheorem{construction}[theorem]{Construction}
\newtheorem{example}[theorem]{Example}
\newtheorem{problem}[theorem]{Problem}

\newcommand{\C}{\mathbb{C}}
\newcommand{\F}{\mathbb{F}}
\newcommand{\K}{\mathbb{K}}
\renewcommand{\S}{\mathbb{S}}
\newcommand{\T}{\mathbb{T}}
\newcommand{\Z}{\mathbb{Z}}

\begin{document}

\title[Semifields, relative difference sets, and bent functions]{Semifields, relative difference sets,\\and bent functions}

\author{Alexander Pott}
\address{Faculty of Mathematics, Otto-von-Guericke University, Universit\"atsplatz~2, 39106 Magdeburg, Germany}
\email{alexander.pott@ovgu.de}

\author{Kai-Uwe Schmidt}
\address{Faculty of Mathematics, Otto-von-Guericke University, Universit\"atsplatz~2, 39106 Magdeburg, Germany}
\email{kaiuwe.schmidt@ovgu.de}

\author{Yue Zhou}
\address{Department of Mathematics and System Sciences, College of Science, National University of Defense Technology, Changsha, China}
\email{yue.zhou.ovgu@gmail.com}

\date{13 January 2014}

\subjclass[2010]{12K10, 05B10, 05B25, 06E30}

\begin{abstract}
Recently, the interest in semifields has increased due 
to the discovery of several new families and progress
in the classification problem. Commutative semifields play an
important role since they are equivalent to certain planar functions
(in the case of odd characteristic) and to modified planar functions
in even characteristic. Similarly, commutative semifields
are equivalent to relative difference sets. The goal of this survey is to
describe the connection between these concepts. Moreover, we shall discuss
power mappings that are planar and consider component functions
of planar mappings, which may be also viewed as projections of
relative difference sets. It turns out that the component functions in the
even characteristic case are related to negabent functions
as well as to $\Z_4$-valued bent functions.
\end{abstract}

\maketitle

\section{Introduction}

Semifields, also called distributive quasifields, have been investigated for almost a century. In this article, we do not want to give a survey about semifields in general, but rather concentrate on commutative semifields. Using the so called Knuth orbit, these are also equivalent to symplectic semifields. It is also not the intention of this article to survey commutative (or symplectic) semifields in general. We refer the reader to the excellent introduction by Lavrauw and Polverino~\cite{LavPol2012}. The aim here is to discuss the equivalence between semifields, relative difference sets, and certain planar functions, which are mappings on finite fields. We hope that the investigation of relative difference sets may also stimulate the research on semifields. For instance, for the investigation of relative difference sets, tools from algebraic number theory are often used, which as far as we know have been only recently applied in the theory of semifields~\cite{PotZho2011}.
\par
There is a difference between semifields in even and odd characteristic, which induces also differences between their corresponding relative difference sets and their planar functions. For example, on the level of the associated relative difference set, the ambient group is elementary abelian in the odd characteristic case, whereas it equals $\Z_4^{\; n}$ (for some~$n$) in the even characteristic case. Using so called projections of relative difference sets, one obtains relative difference sets in subgroups. Such a projection may be viewed as a component function of the planar function associated with the relative difference set. In the case of odd characteristic, the components are $p$-ary bent functions, whereas in the even characteristic case they are (essentially) negabent functions. 
\par
This survey is organized as follows. In Section~\ref{sec:semifields}, we recall some results on semifields, including the connection to projective planes, and discuss the equivalence problem for semifields. In Section~\ref{sec:rds}, we give some background on relative difference sets. In Section~\ref{sec:connections}, we explain the connection between these concepts. In particular we give a partial characterization of those relative difference sets that correspond to commutative semifields. In Sections~\ref{sec:p_odd} and~\ref{sec:p_even}, we look more closely at some examples of semifields, in particular at those that can be described by monomial planar functions. Section~\ref{sec:p_odd} deals with the odd characteristic case and Section~\ref{sec:p_even} with the even characteristic case. In Section~\ref{sec:components}, we briefly investigate the component functions or, equivalently, the projections of the corresponding relative difference sets. We conclude with some open problems in Section~\ref{sec:conclusion}.


\section{Semifields}
\label{sec:semifields}

Roughly speaking, a semifield is a field without associativity for the multiplication. More precisely, a (finite) semifield is defined as follows. 
\begin{definition}
\label{def:semifield}
Let $\S$ be a finite set containing at least two elements. Then $\S$ together with two binary operations $+$ and $\circ$ is a {\em semifield} $(\S,+,\,\circ\,)$ if the following hold: 
\begin{itemize}
\item[(S1)] $(\S,+)$ is an abelian group with identity element $0$.
\item[(S2)] $x\circ(y+z)=x\circ y + x \circ z$ and
$(x+y)\circ z=x\circ z + y \circ z$
for all
$x,y,z\in \S$.
\item[(S3)]
$x\circ y=0$ implies $x=0$ or $y=0$.
\item[(S4)] There is an element $1\ne 0$ such that $1 \circ x=x\circ 1 = x$ for all $x\in \S$.
\end{itemize}
If (S4) is missing, then $\S$ is a \emph{pre-semifield}. If the operation $\circ$ is commutative, then $\S$ is a \emph{commutative} semifield.
\end{definition}
\par
Notice that, since $\S$ is finite, each of the equations
\begin{align*}
a\circ x & = b\\
x\circ a & = b
\end{align*}
has a unique solution $x$ (this would not be true if $\S$ were infinite, in which case the existence of a solution of the above equations has to be added to the axioms for a semifield).
\par
The distributive laws together with (S3) simply say that multiplication from the left and multiplication from the right act as automorphisms of the additive group of the semifield. Let $a$ and $b$ be two nonzero elements in $\S$. Then we find an automorphism $x$ of $(\S,+)$ such that $x\circ a=b$. Hence $a$ and $b$ have the same order and $\S$ must be an elementary abelian $p$-group. This prime number $p$ is also called the {\em characteristic} of the semifield.
\par
Of course, finite fields are semifields. For a current list of known semifields (including some infinite families and some sporadic examples), we refer to \cite{LavPol2012}. We shall see several examples in later sections, which we describe in terms of planar functions. We note that semifields of order $p$ or $p^2$ are necessarily finite fields~\cite{Dic1906}. Moreover, there are only two classes of semifields of order $p^3$~\cite{Men1977}. The situation is getting much more involved for semifields of order $p^4$. 
\par
The fascination of semifields comes from the fact that the elements in $\S$ have two different meanings: Multiplication from the left is a bijective linear mapping on $\S$, viewed as a vector space, but the elements on the right hand side of $\circ$ are just considered to be vectors. This may help to motivate the following definition of isotopy of semifields.
\begin{definition}
Two semifields $(\S,+,\,\circ_s)$ and $(\T,+,\,\circ_t)$ are {\em isotopic} if there are bijective linear mappings $F$, $G$, and $H$ from $\T$ to $\S$ such that
\[
F(x)\circ_s G(y)=H(x\circ_t y)\quad\text{for all $x,y\in \T$}.
\]  
\end{definition}
\par
There are essentially two reasons why isotopism is defined in this unusual way, involving different mappings $F$ and $G$ on the two sides of a product $x\circ y$. One reason is that, as mentioned above, semifield elements on the left are associated with linear mappings, whereas on the right they are interpreted as vectors. Another reason is that every semifield can be used to construct a projective plane and, as we shall see in Proposition~\ref{pro:isotopic-isomorphic}, isotopy just means that the planes are isomorphic. (We do not recall the definition of an isomorphism of an incidence structure here, since we hope it is clear. Otherwise we refer to \cite[Section~I.4]{BetJunLen1999}.) It is known that every pre-semifield is isotopic to a semifield~\cite{LavPol2012}, which is the reason why we restrict ourselves to semifields.
\begin{definition}
A \emph{projective plane} is a point-line incidence structure with the following three properties:
\begin{itemize}
\item[(P1)] Every two different points are contained in a unique line.
\item[(P2)] Every two different lines intersect in exactly
one point.
\item[(P3)] There are four points with the property that no three of them are
contained in a single line.
\end{itemize}
\end{definition}
For background on projective planes, we refer the reader to~\cite{HugPip1973}. If the number of points and lines in a projective plane is finite, then there is a number $n$ (called the \emph{order} of the plane) such that each line contains exactly $n+1$ points and through each point there are exactly $n+1$ lines.
\par
The following are perhaps the two most famous open problems concerning projective planes.
\begin{question}[Prime power conjecture]
\label{que:prime_power_conjecture}
Is the order of a finite projective plane necessarily a prime power?
\end{question}
\par
\begin{question}
\label{que:prime_order_conjecture}
Is a projective plane of prime order unique (up to isomorphism)?
\end{question}
\par
Most researchers believe that Question~\ref{que:prime_order_conjecture} has a positive answer. 
\par
We shall now describe how a projective plane can be constructed from a semifield.
\begin{construction}
\label{con:1}
Let  $(\S,+,\,\circ\,)$ be a semifield and let $m,b\in\S$. We define a point $(x,y)\in \S\times\S$ to be on the line $[m,b]$ if $m\circ x + b=y$. This defines a point-line incidence structure with $|\S|^2$ points and $|\S|^2$ lines.
\end{construction}
\par
The incidence structure in Construction~\ref{con:1} is not quite a projective plane. It is a divisible design, whose definition is recalled below. We shall see in Construction~\ref{con:plane_from_design} how a divisible design gives rise to a projective plane. For background and a substantial collection of results on divisible designs and projective planes, we refer the reader to the two books~\cite{BetJunLen1999}.
\begin{definition}
A {\em divisible design} with parameters $(m,n,k,\lambda)$ is a point-line incidence structure with $mn$ points and $mn$ lines that satisfies the following properties:
\begin{itemize}
\item[(D1)] The point set can be partitioned into $m$ point classes, each of size $n$.
\item[(D2)] The line set can be partitioned into $m$ line classes, each of size $n$.
\item[(D3)] Every two different points not in a common point class are joined by exactly $\lambda$ lines.
\item[(D4)] Every two different lines not in a common line class intersect in exactly $\lambda$ points. 
\item[(D5)] Every line contains exactly $k$ points, and through every point there are exactly $k$ lines.
\end{itemize}
\end{definition}
\par
Note that the definition of a divisible design (as the definition of a projective plane) is symmetric in points and lines. Hence
the incidence structure obtained by interchanging points and lines (the so called \emph{dual incidence structure}) is again a divisible design (a projective plane). 
\par
The following result is readily verified using the defining properties of a semifield.
\begin{proposition}
If $\S$ is a semifield of order $n$, then Construction \ref{con:1} gives a divisible $(n,n,n,1)$ design.
\end{proposition}
\par
We now show how to extend such a divisible design uniquely to a projective plane by adding $n+1$ points and $n+1$ lines.
\begin{construction}
\label{con:plane_from_design}
Let $\mathcal{D}$ be a divisible $(n,n,n,1)$ design. Add a new point $\infty$. For each point class $\mathcal{P}$ of $\mathcal{D}$, add a new line joining the points in $\mathcal{P}$ and $\infty$. For each line class $\mathcal{L}$ of $\mathcal{D}$, add a new point $\infty_{\mathcal{L}}$ to each line in $\mathcal{L}$. Finally, add a new line joining all new points. It is not difficult to see that this incidence structure is a projective plane of order~$n$.
\end{construction}
\par
We note that two isomorphic divisible designs give rise to isomorphic projective planes, but the converse is not necessarily true: One projective plane may be obtained from different (non-isomorphic) divisible designs. This can occur if the automorphism group of the plane is not transitive on lines. 
\par
A \emph{semifield plane} of order $n$ is the projective plane constructed from a semifield of order $n$ via Constructions~\ref{con:1} and~\ref{con:plane_from_design}. The next result shows why semifield isotopy is a natural concept.
\par
\begin{proposition}[\cite{Alb1960}]
\label{pro:isotopic-isomorphic}
Two semifield planes are isomorphic if and only if the corresponding semifields are isotopic.
\end{proposition}
\par
Given a semifield with multiplication $\circ$, another semifield with multiplication $\star$ can be obtained by changing the order of the multiplication, viz $x \star y := y\circ x$. There is a more subtle possibility to obtain more semifields from just one semifield via the Knuth orbit. We refer the reader to \cite{Knu1965} for details.
\par
We conclude this section by showing how a spread can be obtained from a semifield. Let $(\S,+,\,\circ\,)$ be a semifield of characteristic $p$. Let $e_1,\dots,e_n$ be a basis of $\S$, viewed as a vector space over the field with $p$ elements. The mappings $e_i: \S\to \S$ defined by $e_i(v):=e_i\circ v$ are linear and bijective. Moreover, the vector space $W$ of linear mappings on $\S$ spanned by $e_1,\dots,e_n$ has dimension $n$ and consists of $p^n-1$ bijective linear mappings and the zero mapping. As $F$ ranges over $W$, the sets $\{(x,F(x))\: : \: x\in \S\}$ form a set of $p^n$ subspaces of $\S\times \S$, each of dimension $n$. These subspaces together with $\{(0,x)\: :\: x\in \S\}$ form a spread, namely a set of $p^n+1$ $n$-dimensional subspaces of the $2n$-dimensional space $\S\times\S$ intersecting pairwise trivially. A spread corresponding to a semifield is called \emph{semifield spread}.


\section{Relative difference sets}
\label{sec:rds}

We now describe a concept, which seems at a first glance unrelated to semifields.
\begin{definition}
Let $G$ be a group of order $mn$ containing a subgroup $N$ of order $n$. A $k$-subset $R$ of $G$ is called a \emph{relative $(m,n,k,\lambda)$ difference set} (relative to $N$) if the list of nonzero differences $r-r'$ with $r,r'\in R$ contains all elements in $G\setminus N$ exactly $\lambda$ times and no element in $N$. The subgroup $N$ is called the \emph{forbidden subgroup}.
\end{definition}
\par
Sometimes we simply say that $R$ is a difference set relative to $N$. In case that $N=\{0\}$, the definition of a relative difference set coincides with the definition of a difference set in the usual sense. We refer the reader to \cite{BetJunLen1999}, \cite{Pot1995}, \cite{Pot1996} for background.
\begin{example}
(a) The set $\{1,2,4\}\subseteq \Z_8$ is a relative $(4,2,3,1)$ difference set. The forbidden subgroup is the unique subgroup $4\Z_8=\{0,4\}$ of order $2$ in $\Z_8$.
\par
(b) The set $\{(0,0), (0,1), (1,3), (3,0)\}$ is a relative
$(4,4,4,1)$  difference set in $\Z_4\times \Z_4$  with forbidden subgroup 
$2\Z_4\times 2\Z_4$.
\par
(c) The set $\{(0,0), (1,1), (2,1)\}$ is a relative 
difference set in $\Z_3\times \Z_3$ with forbidden subgroup
$\{0\}\times \Z_3$.
\end{example}
\par
We now show how to construct a divisible design from a relative difference set.
\begin{construction}
\label{con:design_from_difference_set}
Let $R$ be a relative $(m,n,k,\lambda)$ difference set in a group $G$ with forbidden subgroup $N$. We construct a divisible $(m,n,k,\lambda)$ design as follows. The points are the elements of $G$ and the lines are the translates $R+g$ with $g\in G$. The point classes are the (right) cosets of $N$. Similarly, the line classes are induced by the (right) cosets of $N$. The group $G$ itself acts via right translation $\tau_g:x\mapsto x+g$ regularly (=sharply transitively) on the points as well as on the lines.
\par
Conversely, a divisible designs that admits an automorphism
group acting regularly on points and lines can be described by a relative difference set. This is a well known fact in 
design theory~\cite[Section~VI.10]{BetJunLen1999}.
\end{construction}
\par
Note that translation $\tau_x$ by elements  $x\in N$ fixes the point class $N$, but not necessarily the right cosets of $N$. The right cosets of $N$ are fixed if and only if $N$ is a normal subgroup. Sometimes normality of $N$ is part of the definition of a relative difference set.
\par
In view of Constructions~\ref{con:plane_from_design} and~\ref{con:design_from_difference_set}, a relative difference set with parameters $(n,n,n,1)$ gives rise to a projective plane of order $n$. It is remarkable that the prime power conjecture (see Question~\ref{que:prime_power_conjecture}) has been solved in the special case that the plane can be described by such a relative difference set in an abelian group. This was proved by Ganley~\cite{Gan1976} for even $n$ (see also \cite{Jun1987}) and by Blokhuis, Jungnickel, and Schmidt~\cite{BloJunSch2002} for odd $n$.
\begin{theorem}[{\cite{Gan1976},~\cite{BloJunSch2002}}]
\label{thm:ppc}
Let $R$ be a relative $(n,n,n,1)$ difference set in an abelian group~$G$. If $n$ is even, then $n=2^m$ (for some $m$) and $G$ is isomorphic to $\Z_4^{\;m}$ and the forbidden subgroup is isomorphic to $\Z_2^{\;m}$. If $n$ is odd, then $n=p^m$ where $p$ is a prime and $G$ contains an elementary abelian subgroup of order $p^{m+1}$. 
\end{theorem}
\par
The following problem remains unsolved.
\begin{question}
Let $q$ be an odd prime power and let $R$ be a relative $(q,q,q,1)$ difference set in an abelian group $G$. Is it true that $G$ must be elementary abelian?
\end{question}
\par
Another question is which planes can be obtained from relative $(n,n,n,1)$ difference sets. We shall see in the next section that with the exception of a single family, all known relative $(n,n,n,1)$ difference sets give rise to semifield planes. Moreover, all known examples of relative difference sets with parameters $(p^2,p^2,p^2,1)$ or $(p,p,p,1)$ describe Desarguesian planes (namely semifield planes constructed from finite fields). This is not a surprise in the latter case since, as mentioned in connection with Question~\ref{que:prime_order_conjecture}, it is conjectured that there is only one plane of prime order (which is necessarily Desarguesian). For planes coming from relative difference sets, this has been proved independently in~\cite{Glu1990},~\cite{Hir1990},~\cite{RonSzo1989}.
\begin{theorem}
A projective plane described by a relative $(p,p,p,1)$ difference set with $p$ prime is Desarguesian. Moreover, the relative difference set is unique up to equivalence. 
\end{theorem}
\par
It is an open question as to whether all projective planes described by relative $(p^2,p^2,p^2,1)$ difference sets (with $p$ prime) are Desarguesian. Since a semifield of order $p^2$ is necessarily a finite field, a putative counterexample cannot be a semifield plane.
\par
In connection with relative difference sets, there are two concepts of equivalence. We call two relative difference sets $R$ and $R'$ in a group $G$ \emph{equivalent} if there is a group automorphism $\varphi$ of $G$ and a group element $g\in G$ such that $R'=\varphi(R)+g$. It is readily verified that equivalent relative difference sets describe isomorphic divisible designs. The converse is in general not true. We call two relative difference sets \emph{isomorphic} if their corresponding divisible designs are isomorphic. Of course, two isomorphic relative difference sets  are equivalent. In general, determining whether two relative difference sets are isomorphic is much more difficult than determining whether they are equivalent.
\par
Relative difference sets have the following nice property.
\begin{proposition}
\label{pro:rds_projection}
Let $R$ be an $(m,n,k,\lambda)$ difference set in an abelian group $G$ relative to $N$. Let $U$ be a subgroup of $N$ of order $u$ and let $\varphi$ denote the canonical epimorphism $G\to G/U$. Then $\varphi(R)$ is a difference set with parameters $(m,n/u,k,\lambda u)$ relative to $N/U$. 
\end{proposition}
\par
With the notation as in Proposition~\ref{pro:rds_projection}, suppose that $N$ can be written as a direct product of abelian groups $N_1,\dots,N_s$. Then we may take a set of coset representatives $g_1,\dots,g_m$ of $N$ in $G$ and describe the new relative difference set using a mapping $g_i\mapsto n_i$ from the set of coset representatives into the group $N$. The element $n_i$ can be written as a product $\prod_{j=1}^s n_i^{(j)}$ with $n_i^{(j)}\in N_j$. We may view the mappings $g_i\mapsto n_i^{(j)}$ as the component functions of the relative difference set. Note that there is much freedom in this construction since we may change the set of coset representatives arbitrarily. Hence it is not the mapping that is important here but the mapping together with the choice of coset representatives. In case that the relative difference set is splitting, there is a canonical set of coset representatives, namely the elements in a group 
theoretic complement of $N$. 


\section{Relative difference sets and semifields}
\label{sec:connections}

We have already seen that a projective plane of order $n$ can be constructed from a relative $(n,n,n,1)$ difference set and also from a semifield of order $n$. The following fundamental theorem shows that a semifield of order $n$ also gives rise to a relative $(n,n,n,1)$ difference set (see \cite{Jun1982} and \cite{GhiJun2003}, for example). 
\begin{theorem}
\label{thm:fund}
Let $(\S,+,\,\circ\,)$ be a semifield of order $n$. Then the set $R=\{(x,x\circ x): x\in\S \}$ is a relative $(n,n,n,1)$ difference set in $G=\{(x,y)\: : \: x,y\in \S \}$, where the group operation on $G$ is given by 
\[
(x,y)\star (x',y') = ( x+x', y+y'+x\circ x').
\]
The forbidden subgroup is $N=\{(0,y): y\in \S\}$.
\end{theorem}
\begin{proof}
It is straightforward to check that $G$ is a group. The inverse element of $(x,y)\in G$ is
\[
(x,y)^{-1}=(-x, -y+x\circ x).
\]
The differences that we can form with elements from $R$ are
\begin{align*}
(a,a\circ a)\star (b,b\circ b)^{-1}&=(a,a\circ a)\star(-b,0)\\
&=(a-b,a\circ a-a\circ b)\\
&=(a-b,a\circ(a-b)),
\end{align*}
whose nonzero values cover all elements in $G\setminus N$ exactly once.
\end{proof}
\par
\begin{remark}
The group $G$ in Theorem~\ref{thm:fund} is commutative if and only if the semifield is commutative. The order of the elements in $G\setminus N$ equals~$p$ if~$p$ is odd and equals~$4$ if $p=2$. In both cases, the forbidden subgroup $N$ is elementary abelian. 
\end{remark}
\par
We may also ask whether the converse of Theorem \ref{thm:fund} is true, namely is it 
possible to construct a semifield starting from a
relative difference set with parameters $(n,n,n,1)$? A partial answer to this question can be given in the case that the ambient group $G$ is either $\Z_4^{\;m}$ or $\Z_p^{\;m}$ for an odd prime $p$.
\par
We first consider the case that $G$ is an abelian $p$-group and $p$ is odd. We represent $G$ as the additive group of $\F_{p^m}\times \F_{p^m}$. It is not hard to verify that every relative $(p^m,p^m,p^m,1)$ difference set in $G$ can be written as
\[
R=\{(x,f(x)):x\in \F_{p^m}\}
\]
for some function $f:\F_{p^m}\to \F_{p^m}$, in which case the forbidden subgroup is $\{0\}\times \F_{p^m}$. The set $R$ is a relative difference set if and only if the equation $f(x+a)-f(x)=b$ has a unique solution $x$ for all $a,b\in\F_{p^m}$ with $a\ne 0$. Functions with this property are called \emph{planar}. Here we give a definition, which is valid for all groups.
\begin{definition}
\label{def:planar}
Let $H$ and $N$ be two groups. A function $f:H\to N$ is \emph{planar} if the mapping $\delta_a:H\to N$ defined by $\delta_a(x)=f(x+a)-f(x)$ is bijective for all nonzero $a\in H$. 
\end{definition}
\par
The following result is readily verified.
\begin{proposition}
\label{pro:rds_from_planar}
Let $H$ and $N$ be two groups of order $n$ and let $R$ be a subset of $H\times N$. Then $R$ is an $(n,n,n,1)$ difference set in $H\times N$ relative to $\{0\}\times N$ if and only if there is a planar function $f:H\to N$ such that
\[
R=\{(x,f(x))\: :\: x\in H\}.
\]
\end{proposition}
\par
In view of Proposition~\ref{pro:rds_from_planar} and Constructions~\ref{con:plane_from_design} and~\ref{con:design_from_difference_set}, every planar function corresponds to a unique projective plane. A partial answer to the question of which relative difference sets describe semifield planes can be given in terms of planar functions, for which we need one more definition. Note that every mapping $f:\F_q\to \F_q$ is a polynomial mapping $f(x)=\sum_{i=0}^{q-1} a_ix^i$ for some uniquely determined $a_0,\dots,a_{q-1}\in \F_q$.
\begin{definition}
Let $p$ be a prime. A polynomial $f\in\F_{p^m}[x]$ (and also the corresponding mapping) is called \emph{Dembowski-Ostrom} if, for some $a_{i,j}\in\F_{p^m}$, $$f(x)=\sum_{0\le i\le j<m} a_{i,j}x^{p^i+p^j}$$ for odd $p$ or $$f(x)=\sum_{0\le i<j<m} a_{i,j}x^{p^i+p^j}$$ for $p=2$. The polynomial $f$ is \emph{affine} if $f(x)=\sum_{i=0}^{m-1}c_ix^{p^i}+d$ for some $c_i,d\in \F_{p^m}$. An \emph{affine Dembowski-Ostrom polynomial} is a sum of a Dembowski-Ostrom polynomial and an affine polynomial.
\end{definition}
\par
Our next result, which is essentially~\cite[Theorem 3.3]{CouHen2008}, gives the promised partial answer to the question of which relative difference sets
describe semifield planes.
\begin{theorem}
\label{thm:semifield_DO_odd}
Let $q$ be an odd prime power and let $f:\F_q\to \F_q$ be a mapping. If $f$ is an affine Dembowski-Ostrom planar function, then the corresponding projective plane is a semifield plane, where the semifield is isotopic to the (commutative) pre-semifield $(\F_q,+,\,\circ\,)$ whose multiplication is given by $x\circ y=f(x+y)-f(x)-f(y)+f(0)$. Conversely, if a projective plane is a semifield plane corresponding to a commutative semifield $(\S,+,\,\star \,)$, then $f(x)=x\star x$ is a Dembowski-Ostrom planar function.
\end{theorem}
\par
Now suppose that $f:\F_q\to\F_q$ is a planar function, but not an affine Dembowski-Ostrom polynomial. Is it possible that the corresponding projective plane is a semifield plane (where the semifield multiplication is possibly not related in an obvious way to the function $f$)? As far as we know, this problem remains open.
\par
Only one family of planar functions not of Dembowski-Ostrom type is known. It is also known that the corresponding projective planes are not semifield planes.
\begin{theorem}[\cite{CouMat1997}]
\label{thm:non-DO}
Let $f:\F_{3^m}\to \F_{3^m}$ be a mapping defined by $f(x)=x^{(3^k+1)/2}$. Then $f$ is planar if $\gcd(k,2m)=1$. If $3\le k<m-1$, the corresponding projective plane is not a semifield plane.
\end{theorem}
\par
The following nice observation is contained in \cite{WenZen2012}.
\begin{proposition}[{\cite[Theorem 2.3]{WenZen2012}}]
\label{pro:planar_2-1}
Let $q$ be an odd prime power and let $f\in\F_q[x]$ be a Dembowski-Ostrom polynomial. Then $f$ is
planar if and only if $f$ is a $2$-to-$1$ mapping, namely $f(x)=a$ has $0$ or $2$ solutions in $x$ for all nonzero $a\in\F_q$.
\end{proposition}
\par
\begin{example}
\label{exa:planar_example}
The mapping $x\mapsto x^{10}\pm x^6-x^2$ on $\F_{3^m}$ is planar for odd $m$ (and also Dembowski-Ostrom). These functions have been found \cite{CouMat1997} and \cite{DinYua2006}. The planarity is easily verified using Proposition~\ref{pro:planar_2-1}. First observe that the polynomial $y^5\pm y^3-y$ is a Dickson polynomial, which is a permutation polynomial in $\F_{3^m}[y]$ if and only if $m$ is odd~\cite[Theorem~6.17]{LidNie1997}. Hence, after replacing $y$ by $x^2$, we obtain a $2$-$1$ Dembowski-Ostrom mapping on $\F_{3^m}$ for odd $m$.
\end{example}
\par
Next we consider the case that the ambient group $G$ is $\Z_4^{\;m}$. Unlike in the case that $G$ is elementary abelian, there is no canonical way to represent $G$ by the additive group of a finite field. One may use the Galois ring $\operatorname{GR}(4,m)$ as in~\cite{SchZho2013}, which has the advantage that the Galois ring multiplication can be used, but we prefer a different (equivalent) approach. 
\par
We represent the group $\Z_4^{\; m}$ as $\F_{2^m}\times \F_{2^m}$ with the group operation
\begin{equation}
(x,y)\star(x',y')=(x+x',y+y'+x\cdot x').   \label{eqn:group_operation}
\end{equation}
Since this group is not a direct product of two groups of order $2^m$, we cannot use Proposition~\ref{pro:rds_from_planar} to construct relative $(2^m,2^m,2^m,1)$ difference sets in this group. However, observe that every relative $(2^m,2^m,2^m,1)$ difference set in $\F_{2^m}\times\F_{2^m}$ can still be written as
\[
R=\{(x,f(x))\, : \, x \in \F_{2^m}\}
\]
for some function $f:\F_{2^m}\to\F_{2^m}$, in which case the forbidden subgroup is $\{0\}\times\F_{2^m}$. It is then readily verified that $R$ is a relative difference set if and only if the equation
\[
f(x+a)-f(x)+a\cdot x=b
\]
has a unique solution $x$ for all $a,b\in\F_{2^m}$ with $a\ne 0$. We therefore modify the definition of a planar function from $\F_{2^m}$ to $\F_{2^m}$, which reflects the structure of the ambient group.
\begin{definition}
\label{def:planar_even}
A function $f:\F_{2^m}\to \F_{2^m}$ is \emph{planar} if $x\mapsto f(x+a)+f(x)+ax$ is a permutation on $\F_{2^m}$ for all nonzero $a\in \F_{2^m}$. 
\end{definition}
\par
There will be no confusion of this definition with Definition~\ref{def:planar} since every function $f:\F_{2^m}\to\F_{2^m}$ satisfies
\[
f(x+a)-f(x)=f((x+a)+a)-f(x+a)\quad\text{for all $x,a\in\F_{2^m}$},
\]
and hence planar functions from $\F_{2^m}$ to $\F_{2^m}$ according to Definition~\ref{def:planar} cannot exist. A trivial example of a planar function $f:\F_{2^m}\to\F_{2^m}$ is $f(x)=0$. More generally, every affine polynomial $f\in\F_{2^m}[x]$ induces a planar function on $\F_{2^m}$.
\par
We have now established the following result, which is essentially contained in~\cite{Zho2013}.
\begin{theorem}[\cite{Zho2013}]
\label{thm:z4rep}
Let $R$ be a subset of the group $\F_{2^m}\times \F_{2^m}$ whose operation is given in~\eqref{eqn:group_operation}. Then $R$ is a $(2^m,2^m,2^m,1)$ difference set relative to $\{0\}\times\F_{2^m}$ if and only if there is a planar function $f:\F_{2^m}\to\F_{2^m}$ such that
\[
R=\{(x,f(x))\, : \, x \in \F_{2^m}\}.
\]
\end{theorem}
\par
Theorem~\ref{thm:z4rep} should be compared with Proposition~\ref{pro:rds_from_planar}. Again, by Constructions~\ref{con:plane_from_design} and~\ref{con:design_from_difference_set}, every planar function according to Definition~\ref{def:planar_even} corresponds to a unique projective plane. We also have the following counterpart of Theorem~\ref{thm:semifield_DO_odd}.
\begin{theorem}
\label{thm:semifield_DO_even}
Let $q$ be a power of $2$ and let $f:\F_q\to \F_q$ be a mapping. If $f$ is an affine Dembowski-Ostrom planar function, then the corresponding projective plane is a semifield plane, where the semifield is isotopic to the (commutative) pre-semifield $(\F_q,+,\,\circ\,)$ whose multiplication is given by $x\circ y=f(x+y)+f(x)+f(y)+f(0)+xy$. Conversely, if a projective plane is a semifield plane corresponding to a commutative semifield $(\S,+,\,\star\,)$, then $f(x)=x\star x$ is a Dembowski-Ostrom planar function.
\end{theorem}
\par
All known planar functions from $\F_{2^m}$ to $\F_{2^m}$ are induced by affine Dembowski-Ostrom polynomials. Hence, in view of Theorem~\ref{thm:semifield_DO_even}, all known planar functions from $\F_{2^m}$ to $\F_{2^m}$ produce semifield planes. This is contrary to the case of planar functions from $\F_{p^m}$ to $\F_{p^m}$ for $p$ an odd prime, where we have the exceptional example given in Theorem~\ref{thm:non-DO}.


\section{Planar functions in odd characteristic}
\label{sec:p_odd}

We have seen in the previous section that every commutative semifield can be described by a planar function. We do not claim that such a description is natural and indeed the planar function of a commutative semifield can look quite cumbersome. However, interesting examples of commutative semifields have been found using Dembowski-Ostrom planar functions (see \cite{ZhaKyuWan2009} and \cite{ZhoPot2011}, for example). We have already seen examples of such planar functions in Example~\ref{exa:planar_example}. 
\par
In what follows, we let $p$ be an odd prime and consider the simplest polynomial mappings on $\F_{p^m}$, namely monomial mappings $x\mapsto cx^d$. It is readily verified that in order to study planar monomials $cx^d$ in $\F_{p^m}[x]$, we can without loss of generality restrict ourselves to the case $c=1$ and $d<p^m$ and $p\nmid d$. The only known examples of planar monomial mappings $x\mapsto x^d$ on $\F_{p^m}$ with $d<p^m$ and $p\nmid d$ are given in Table~\ref{tab:pp_odd}.
\begin{table}[ht]
\centering
\caption{Known planar monomial mappings $x^d$ on $\F_{p^m}$}
\vspace{1ex}
\label{tab:pp_odd}
\begin{tabular}{c|c|c|c}
\hline
$d$                & $p$ & condition                     & reference\\ \hline\hline
$2$                & odd & none                          & folklore, Desarguesian plane\\[2ex]
$p^k+1$            & odd & $\dfrac{m}{\gcd(k,m)}$ is odd & commutative Albert semifields \cite{Alb1961}\\[2ex]
$\dfrac{p^k+1}{2}$ & $3$ & $\gcd(k,2m)=1$                & Coulter-Matthews \cite{CouMat1997}\\[1ex]\hline
\end{tabular}
\end{table}
\par
It is sometimes conjectured that the list of examples given Table~\ref{tab:pp_odd} is exhaustive. We believe that this is difficult to prove. It is however possible to get a slightly weaker result, for which we need the following definition.
\begin{definition}
Let $p$ be a prime. A monomial $x^d$ in $\F_p[x]$ is \emph{exceptional planar} if $x^d$ induces a planar function on infinitely many extensions of~$\F_p$.
\end{definition}
\par
Notice that the mappings in Table~\ref{tab:pp_odd} are all exceptional. Indeed, it has been proved in~\cite{Led2012} and~\cite{Zie2013} that there are no further examples (the case that $p\mid d-1$ is handled in~\cite{Led2012} and the remaining cases are settled in~\cite{Zie2013}).
\begin{theorem}[{\cite{Led2012},~\cite{Zie2013}}]
The only exceptional planar monomials $x^d$ in $\F_p[x]$ with $p\nmid d$ are those given in Table~\ref{tab:pp_odd}. 
\end{theorem}


\section{Planar functions in characteristic $2$}
\label{sec:p_even}

In this section, we give some examples of planar functions from $\F_{2^m}$ to $\F_{2^m}$ according to Definition~\ref{def:planar_even}. We have already seen that every affine polynomial in $\F_{2^m}[x]$ induces a planar function from $\F_{2^m}$ to $\F_{2^m}$. These are trivial examples. A large family of nontrivial examples were given in~\cite{Zho2013}, which is based on a construction for commutative semifields due to Kantor~\cite{Kan2003}.
\begin{theorem}[{\cite[Example~2.2]{Zho2013}}]
\label{thm:Kantor_planar}
Assume that we have a chain of finite fields $\K=\K_0\supset\K_1\supset\cdots\supset\K_n$ of characteristic $2$ with $[\K:\K_n]$ odd. Let $\operatorname{tr}_i$ be the relative trace from $\K$ to $\K_i$. Then, for all nonzero $\zeta_1,\dots,\zeta_n\in\K$, the mapping $f:\K\to\K$ given by
\[
f(x)=\left( x\sum_{i=1}^n \operatorname{tr}_i(\zeta_ix)\right)^2
\]
is planar.
\end{theorem}
\par
Kantor~\cite{Kan2003} also gives a lower bound on the number of non-isomorphic projective planes constructed by the planar functions in Theorem~\ref{thm:Kantor_planar}. 
\par
In what follows, we consider planar monomial mappings $x\mapsto cx^d$ on $\F_{2^m}$. Unlike in the case of odd characteristic, the planarity of this function can depend on the choice of the coefficient $c$. We can however assume without loss of generality that $d<2^m$. We are interested in those exponents $d<2^m$ such that $x\mapsto cx^d$ is planar on $\F_{2^m}$ for some nonzero $c\in\F_{2^m}$. The only known such exponents are listed in Table~\ref{tab:pp_even}.
\begin{table}[ht]
\centering
\caption{Known exponents $d$ such that $cx^d$ is planar on $\F_{2^m}$ for some $c\in\F_{2^m}$}
\label{tab:pp_even}
\vspace{1ex}
\begin{tabular}{c|c|c}
\hline
$d$          &   condition & reference\\ \hline\hline
$2^k$        &   none      & trivial\\
$2^k+1$      &   $m=2k$    & \cite{SchZho2013}\\
$4^k(4^k+1)$ &   $m=6k$    & \cite{SchZie2013}\\\hline
\end{tabular}
\end{table} A full characterization of those $c\in\F_{2^m}$ for which these monomials are planar on $\F_{2^m}$ is also given in~\cite{SchZie2013}. It can be shown that the planes corresponding to the planar functions identified in Table~\ref{tab:pp_even} are all Desarguesian.
\par
It has been conjectured in~\cite{SchZho2013} that the list provided in Table~\ref{tab:pp_even} is exhaustive. As in the case of odd characteristic, we believe that this is difficult to prove. We are therefore interested in classifying \emph{exceptional planar exponents}, namely positive integers $d$ such that $x\mapsto cx^d$ is planar on $\F_{2^m}$ for some nonzero $c\in\F_{2^m}$ and infinitely many $m$. It was shown in~\cite{SchZho2013} that, if $d$ is an odd exceptional planar exponent, then $d=1$. The even exceptional planar exponents have been classified in~\cite{MulZie2013}. In fact, the following sharper result was proved in \cite{MulZie2013}.
\begin{theorem}[{\cite[Theorem~1.1]{MulZie2013}}]
Let $d$ be a positive integer such that $d^4\le 2^m$ and let $c\in\F_{2^m}$ be nonzero. Then the function $x\mapsto cx^d$ is planar on $\F_{2^m}$ if and only if $d$ is a power of $2$.
\end{theorem}


\section{Component functions of planar functions}
\label{sec:components}

In this section we study the component functions corresponding to a planar function. When $p$ is an odd prime, the component functions of a planar function on $\F_{p^m}$ are $p$-ary bent functions. These are well-studied objects. We therefore study the component functions corresponding to a planar function on $\F_{2^m}$, according to Definition~\ref{def:planar_even}. Identifying such a planar function with a $(2^m,2^m,2^m,1)$ difference set in $\Z_4^{\;m}$ relative to $2\Z_4^{\;m}$, the component functions are obtained by a projection with respect to a subgroup of $2\Z_4^{\;m}\cong\Z_2\times\cdots\times\Z_2$ of order $2^{m-1}$ (see Proposition~\ref{pro:rds_projection}). Thus the component functions correspond to relative $(2^m,2,2^m,2^{m-1})$ difference sets in $\Z_4\times\Z_2^{\;m-1}$. We represent this group as follows. Let $B$ be a symmetric bilinear form on $\F_2^{\;m}$, write
\[
G=\{(x,y):x\in \F_2^{\;m}, y\in \F_2\},
\]
and define an operation on $G$ via
\begin{equation}
(x,y)\ast(x',y')=(x+x', y+y'+B(x,x')).   \label{eqn:G_op}
\end{equation}
\begin{proposition}
With the notation as above, $(G,\,\ast\,)$ is an abelian group isomorphic to $\Z_4\times \Z_2^{\,m-1}$ if $B$ is nonalternating and isomorphic to $\Z_2^{\;m+1}$ if $B$ is alternating. 
\end{proposition}
\begin{proof}
It is immediate that $(G,\,\ast\,)$ is abelian. We have $(x,y)\ast(x,y)=(0,B(x,x))$. Hence, the nonzero elements in $G$ have order $2$ or $4$. If $B$ is alternating, then every element in $G$ has order $2$ and $G$ is isomorphic to $\Z_2^{\;m+1}$. If $B$ is nonalternating, then $B(x,x)$ is a nontrivial linear form, from which we see that exactly half of the elements in $G$ have order $2$ and therefore $G$ is isomorphic to $\Z_4\times \Z_2^{\,m-1}$.
\end{proof}
\par
The following result characterizes the relative $(2^m,2,2^m,2^{m-1})$ difference sets in~$G$.
\begin{theorem}
\label{thm:rds_negabent}
Let $R$ be a subset of $G$ whose operation is given in~\eqref{eqn:G_op}. Then $R$ is a $(2^m,2,2^m,2^{m-1})$ difference set in $G$ relative to $N=\{(0,y):y\in\F_2\}$ if and only if there is a function $f:\F_2^{\;m}\to\F_2$ such that
\[
f(x+a)+f(x)+B(a,x)=b
\]
has $2^{m-1}$ solutions for all $b\in\F_2$ and all nonzero $a\in \F_2^{\;m}$ and
\begin{equation}
R=\{(x,f(x))\,:\,x\in \F_2^{\;m}\}.   \label{eqn:R_from_f}
\end{equation}
\end{theorem}
\begin{proof}
Note that the inverse of $(x,y)\in G$ is given by
\[
(x,y)^{-1}=(x,y+B(x,x)).
\]
We therefore obtain, for all $x,a\in\F_2^{\;m}$,
\begin{equation}
(x+a,f(x+a))\ast(x,f(x))^{-1}=(a,f(x+a)+f(x)+B(a,x)).   \label{eqn:differences}
\end{equation}
We now readily verify that, if $f$ has the properties stated in the theorem, then $R$ is a $(2^m,2,2^m,2^{m-1})$ difference set relative to $N$. Conversely, if $R$ is such a relative difference set, then there is some function $f:\F_2^{\;m}\to\F_2$ such that~\eqref{eqn:R_from_f} holds. From~\eqref{eqn:differences} we then verify that $f$ must have the properties stated in the theorem.
\end{proof}
\par
Let $f$ be a function from $\F_2^{\;m}$ to $\F_2$ and write $R=\{(x,f(x))\,:\,x\in \F_2^{\;m}\}$. It is well known that relative difference sets are subsets of a group with certain character values. In particular, $R$ is a $(2^m,2,2^m,2^{m-1})$ difference set in $G$ relative to $N=\{(0,y):y\in\F_2\}$ if and only if 
\[
\Big|\sum_{z\in R}\chi(z)\Big|^2=2^m
\]
for every character $\chi$ of $G$ that is nontrivial on $N$. The characters of $G$ depend on the choice of the bilinear form $B$. Take the standard bilinear form defined via the scalar product
\[
B(x,y)=\langle x,y\rangle.
\]
Then the characters of $G$ are as follows. For $a\in\F_2^{\;m}$, the functions $\chi_a:G\to\C$ defined by 
\[
\chi_a(x,y)=(-1)^{\langle a, x \rangle}
\]
are the $2^m$ characters of $G$ that are trivial on $N$. Define $\gamma:G\to\C$ by
\[
\gamma(x,y)=i^{w(x)}(-1)^y,
\]
where $w(x)$ denotes the (Hamming) weight of the vector $x\in\F_2^{\;m}$ and $i=\sqrt{-1}$. It is readily verified that $\gamma$ is indeed a homomorphism. Therefore, the $2^m$ characters of $G$ that are nontrivial on $N$ are $\chi_a\cdot\gamma$ for $a\in\F_2^{\;m}$.
\par
Since
\[
\sum_{z\in R}(\chi_a\cdot\gamma)(z)=\sum_{x\in \F_2^m}(-1)^{\langle x, a\rangle+f(x)}i^{w(x)},
\]
we find that $R$ is a $(2^m,2,2^m,2^{m-1})$ difference set in $G$ relative to $N$ if and only if
\[
\bigg|\sum_{x\in \F_2^m}(-1)^{\langle x, a\rangle+f(x)}i^{w(x)}\bigg|^2=2^m
\]
for all $a\in\F_2^{\;m}$. Functions $f$ with this property have been called \emph{negabent} in the literature~\cite{ParPot2007}. One can also show that $\Z_4$-valued bent functions~\cite{Sch2009} are equivalent to such relative difference sets, hence also equivalent to negabent functions.
\par
We remark that, while the notions of negabent functions and $\Z_4$-valued bent functions have been introduced only fairly recently, the underlying relative difference sets have been studied before, as the following construction from~\cite{AraJunPot1990} shows.
\begin{theorem}
\label{thm:negabent_from_bent}
Let $G$ be a group and let $D$ and $E$ be two difference sets (in the usual sense) in $G$. Then the set
\[
\{0\} \times D \ \cup\  \{1\} \times E \  \cup \  \{2\} \times (G\setminus D)
\  \cup\  \{3\} \times (G\setminus E)
\]
is a relative $(2|G|,2,2|G|,|G|)$ difference set in $\Z_4\times G$ relative to $2\Z_4\times\{0\}$.
\end{theorem}
\par
Note that every bent function on $\F_2^{\; m}$ gives rise to a difference set in $\Z_2^{\; m}$. Hence we can use Theorem~\ref{thm:negabent_from_bent} to construct from every bent function a relative difference set in $\Z_4\times \Z_2^{\; m}$ that corresponds to a negabent function on $\F_2^{\;m+1}$. Since, for a bent function on $\F_2^{\; m}$, $m$ is necessarily even, we obtain negabent functions on $\F_2^{\; m}$ with $m$ odd.


\section{Concluding remarks and open problems}
\label{sec:conclusion}

We summarize some open problems related to the content of this paper. We note that there are many problems related to semifields and bent functions, which we do not want to recall here; we just want to restrict ourselves to problems which arise from the difference set point of view.
\begin{problem}
Improve the bound on the rank of an abelian group in Theorem \ref{thm:ppc} 
containing a relative $(p^m,p^m,p^m,1)$ difference set if $p$ is an odd prime.
\end{problem}
\par
Related to this problem one may ask whether a result similar to Theorem \ref{thm:ppc} holds for nonabelian groups. One may also ask whether it is possible to relax the condition that $\lambda=1$ to a small value of $\lambda$. The case $\lambda=2$ is discussed in \cite{Hir2010}. 
\par
In case of odd characteristic, we know one example of a planar function that is not Dembowski-Ostrom (see Table \ref{tab:pp_odd}). In the even characteristic case, such an example is not known.
\begin{problem}
Is it possible to find planar functions in characteristic $2$ that are not of Dembowski-Ostrom type?
\end{problem}
\par
As we have seen, difference sets are natural descriptions of projective planes. Some (but not all) interesting substructures of planes have nice interpretations when the plane is described using a difference set (unitals, subplanes, ovals, arcs, blocking sets \cite{ResGhiJun2002}, \cite{BilKor1989}, \cite{GhiJun2003b}, \cite{GhiJun2006}, \cite{Ho1997}, for example). Typically, the classical difference set representation of a plane (namely a Singer cycle~\cite{Sin1938})  or planar functions in odd characteristic have been used. We believe that more interpretations can be found using the planar functions in even characteristic described here.
\begin{problem}
Is it possible to describe substructures of the Desarguesian projective plane easily in terms of planar functions or in terms of the corresponding relative difference set in $\Z_4^{\; m}$?
\end{problem}


\providecommand{\bysame}{\leavevmode\hbox to3em{\hrulefill}\thinspace}
\providecommand{\MR}{\relax\ifhmode\unskip\space\fi MR }
\providecommand{\MRhref}[2]{%
  \href{http://www.ams.org/mathscinet-getitem?mr=#1}{#2}
}
\providecommand{\href}[2]{#2}

\end{document}